\documentclass[11pt]{article}
\usepackage{url}
\usepackage{makeidx}  
\usepackage{graphicx}
\usepackage{amsmath}
\usepackage{amsthm}
\usepackage{amssymb}
\usepackage{amsfonts}
\usepackage{amscd}
\usepackage[matrix,arrow,curve]{xy}
\usepackage{mathabx}
\usepackage{enumitem}
\usepackage{mathrsfs}
\usepackage{upgreek}
\usepackage{bm}

\newcommand\hide[1]{\commented{gray}{Hidden:}{#1}}
\renewcommand\hide[1]\empty

\newcommand\blfootnote[1]{%
  \begingroup
  \renewcommand\thefootnote{}\footnote{#1}%
  \addtocounter{footnote}{-1}%
  \endgroup
}

\begin{document}

\author{Denis~I.~Saveliev%
\blfootnote{
This research was partially supported 
by the MSHE project No.~FSMG-2024-0048.
}
\blfootnote{
\textit{Mathematics Subject Classification~2010}: 
Primary: 
14A22, 
17A01, 
54D80, 
54H13; 
Secondary: 
03C05, 
03C15, 
03C60, 
05D10, 
05E15, 
22A22, 
22A30, 
54D35, 
54F45, 
54H11. 
}
}
\title{Hindman's finite sums theorem and \\
its application to topologizations of algebras}
\date{}

\maketitle

\theoremstyle{plain}
\newtheorem{thm}{Theorem}[subsection]
\newtheorem*{tm}{Theorem}

\newtheorem{lmm}[thm]{Lemma}
\newtheorem*{lm}{Lemma}

\newtheorem{coro}[thm]{Corollary}
\newtheorem*{cor}{Corollary}

\theoremstyle{definition}
\newtheorem{dfn}[thm]{Definition}
\newtheorem*{df}{Definition}

\newtheorem{exmp}[thm]{Example}
\newtheorem*{exm}{Example}

\newtheorem{exmps}[thm]{Examples}
\newtheorem*{exms}{Examples}

\newtheorem{ntt}[thm]{Notation}
\newtheorem*{nt}{Notation}

\newtheorem{fact}[thm]{Fact}
\newtheorem*{fct}{Fact}

\newtheorem{facts}[thm]{Facts}
\newtheorem*{fcts}{Facts}

\newtheorem{prbl}[thm]{Problem}
\newtheorem*{prb}{Problem}

\newtheorem{rmrk}[thm]{Remark}
\newtheorem*{rmk}{Remark}

\newtheorem{rmrks}[thm]{Remarks}
\newtheorem*{rmks}{Remarks}

\newtheorem{qn}[thm]{Question}
\newtheorem*{q}{Question}

\newtheorem{qns}[thm]{Questions}
\newtheorem*{qs}{Questions}

\newcommand{\Fr}{\mathop{\mathrm {Fr}}\nolimits}
\newcommand{\ind}{\mathop{\mathrm {ind}}\nolimits}
\newcommand{\Ind}{\mathop{\mathrm {Ind}}\nolimits}

\newcommand{\spec}{\mathop{\mathrm {spec}}\nolimits}
\newcommand{\specm}{\mathop{\mathrm {specm}}\nolimits}
\newcommand{\fix}{\mathop{\mathrm {fix}}\nolimits}
\newcommand{\pcf}{\mathop{\mathrm {pcf}}\nolimits}
\newcommand{\crit}{\mathop{\mathrm {crit}}\nolimits}

\newcommand{\dom}{ {\mathop{\mathrm {dom\,}}\nolimits} }
\newcommand{\ran}{ {\mathop{\mathrm{ran\,}}\nolimits} }
\newcommand{\cf}{ {\mathop{\mathrm {cf\,}}\nolimits} }
\newcommand{\club}{ {\mathop{\mathrm {club}}\nolimits} }
\newcommand{\cl}{ {\mathop{\mathrm {cl\,}}\nolimits} }
\newcommand{\cof}{ {\mathop{\mathrm {cof\,}}\nolimits} }
\newcommand{\add}{ {\mathop{\mathrm {add\,}}\nolimits} }
\newcommand{\sat}{ {\mathop{\mathrm {sat\,}}\nolimits} }
\newcommand{\tc}{ {\mathop{\mathrm {tc\,}}\nolimits} }
\newcommand{\unif}{ {\mathop{\mathrm {unif\,}}\nolimits} }
\newcommand{\pr}{ {\mathop{\mathrm {pr\/}}\nolimits} }

\newcommand{\uhr}{\!\upharpoonright\!}
\newcommand{\lra}{ {\leftrightarrow} }
\newcommand{\ot}{ {\mathop{\mathrm {ot\,}}\nolimits} }
\newcommand{\ol}{\overline}
\newcommand{\cnc}{ {^\frown} }
\newcommand{\image}{\/``\,}
\renewcommand{\bm}{\boldsymbol}
\newcommand{\scc}{\bm\upbeta}

\newcommand{\FS}{ {\mathop{\mathrm {FS\,}}\nolimits} }
\newcommand{\FP}{ {\mathop{\mathrm {FP\,}}\nolimits} }
\newcommand{\Lim}{\mathrm{Lim}}
\newcommand{\Reg}{\mathrm{Reg}}
\newcommand{\RK}{ {\mathrm {RK}} }

\newcommand{\ACA}{ {\mathrm {ACA}} }
\newcommand{\RCA}{ {\mathrm {RCA}} }

\newcommand{\CA}{ {\mathrm {CA}} }
\newcommand{\PC}{ {\mathrm {PC}} }
\newcommand{\PRA}{ {\mathrm {PRA}} }
\newcommand{\T}{ {\mathrm T} }
\newcommand{\TA}{ {\mathrm {TA}} }
\newcommand{\PA}{ {\mathrm {PA}} }
\newcommand{\KP}{ {\mathrm {KP}} }
\newcommand{\Z}{ {\mathrm Z} }
\newcommand{\ZF}{ {\mathrm {ZF}} }
\newcommand{\ZFA}{ {\mathrm {ZFA}} }
\newcommand{\ZFC}{ {\mathrm {ZFC}} }
\newcommand{\AC}{ {\mathrm {AC}} }
\newcommand{\GC}{ {\mathrm {GC}} }
\newcommand{\DC}{ {\mathrm {DC}} }
\newcommand{\AD}{ {\mathrm {AD}} }
\newcommand{\CH}{ {\mathrm {CH}} }
\newcommand{\GCH}{ {\mathrm {GCH}} }
\newcommand{\APr}{ {\mathrm {APr}} }
\newcommand{\ASp}{ {\mathrm {ASp}} }

\begin{abstract}
The first part of the paper is a~brief overview of 
Hindman's finite sums theorem, its prehistory and a~few 
of its further generalizations, and a~modern technique 
used in proving these and similar results, which is based 
on idempotent ultrafilters in ultrafilter extensions 
of semigroups. The second, main part of the paper is 
devoted to the topologizability problem of a~wide class 
of algebraic structures called polyrings; this class 
includes Abelian groups, rings, modules, algebras over 
a~ring, differential rings, and others. We show that 
the Zariski topology of such an algebra is always 
non-discrete. Actually, a~much stronger fact holds: if 
$K$~is an infinite polyring, $n$~a~natural number, and 
a~map~$F$ of $K^n$ into~$K$ is defined by a~term in 
$n$~variables, then $F$~is a~closed nowhere dense subset 
of the space~$K^{n+1}$ with its Zariski topology. 
In particular, $K^n$~is a~closed nowhere dense subset 
of~$K^{n+1}$. The proof essentially uses a~multidimensional 
version of Hindman's finite sums theorem.  
The third part of the paper lists some problems 
concerning topologization of various algebraic structures, 
their Zariski topologies, and related questions. 

\hide{
This paper is an extended version of the lecture at 
{\it Journ{\'e}es sur les Arithm{\'e}tiques Faibles~36: 
{\`a}~l'occasion du 70{\`e}me anniversaire de Yuri 
Matiyasevich\/}, delivered on 7th July, 2017, 
in Saint Petersburg.
}

\end{abstract}


\section{Ramsey theory of finite sums}

In this section we shortly recall some basic things
related to the famous Hindman finite sums theorem, 
including some historic information, a~modern technique 
used in proving results of such kind, which is based 
on idempotent ultrafilters in ultrafilter extensions 
of semigroups. Then we formulate a~couple of 
generalizations of the theorem, one of which will 
be essential for the proof of our main result 
in the next section of this paper.

\subsection{Algebraic Ramsey theory}

As well-known, {\it Ramsey theory\/} is a~vast area 
having various aspects including purely combinatorial 
and set-theoretic; for general information we refer 
the reader to the classical textbooks 
\cite{Graham et al} and (from a~more set-theoretic 
perspective)~\cite{Erdos et al}. Informally speaking, 
a~statement can be considered as Ramsey-theoretic iff
it~has the form ``Any small partition of a~large 
structure has a~large part.''

In this paper, we shall be interested in {\it infinite\/} 
Ramsey theory, where the weakest meaning of ``small'' 
naturally is ``finite''. An easy observation (see 
e.g.~\cite{Hindman Strauss}, Theorem~5.7) is that 
in this case, ``large'' can be always understood as 
``ultrafilter large'':

\begin{thm}\label{Ramsey}
Let $X$ be a~set and $\mathcal A$ a~family of its subsets. 
The following are equivalent:
\begin{itemize}
\setlength\itemsep{-0.39em}
\item[(i)] 
any finite partition of~$X$ has a~part~$S$ such~that 
there is $A\in\mathcal A$ with $A\subseteq S$;
\item[(ii)] 
there exists an ultrafilter~$u$ over~$X$ such that 
for any $S\in u$ there is $A\in\mathcal A$ with $A\subseteq S$. 
\end{itemize}
\end{thm}

Furthermore, we shall discuss here an {\it algebraic\/} 
aspect of infinite Ramsey theory, where (again informally
speaking) ``large'' means ``having a~rich algebraic 
structure''. More specifically, for our purposes, ``large'' 
will mean ``having many finite sums (or products)''. 
The answer to what is an ultrafilter counterpart in this 
case, will be given a~bit later, in Section~1.3.


\subsection{Finite sums}

As usually, $\mathbb N$~is the set of natural numbers,
which are understood as finite ordinals, and 
$\omega=\mathbb N$ is the first infinite ordinal.

\begin{ntt}\label{fs-fp}
1. 
Given a~(finite or infinite) sequence 
$x_0,x_1,\ldots,x_i,\ldots$ in~$\mathbb N$, let 
$$\FS(x_i)_i$$ denote the set of finite sums 
$x_{i_0}+x_{i_1}+\ldots+x_{i_n}$ 
for all $n$ and $i_0<i_1<\ldots<i_n$ in~$\mathbb N$.

2. 
If $X$~is a~(multiplicatively written) semigroup, 
for a~sequence $x_0,x_1,\ldots,x_i,\ldots$ in~$X$, 
let $$\FP(x_i)_i$$ denote the set of finite products 
$x_{i_0}\cdot x_{i_1}\cdot\ldots\cdot x_{i_n}$ 
for all $n$ in~$\mathbb N$ and $i_0<i_1<\ldots<i_n$ in~$X$.
\end{ntt}

Let us point out that if the sequence~$(x_i)_i$ was 
injective, then all summands in $\FS(x_i)_i$ (respectively, 
factors in $\FP(x_i)_i$) involve distinct elements. Also 
notice that if the semigroup~$X$ is non-commutative, then 
the increasing ordering of $i_0,i_1,\ldots,i_n$ is essential.

\begin{exmp}\label{fs}
$\FS(1,3,5)=\{1,3,4,5,6,8,9\}$. 
\end{exmp}

The historically first Ramsey-theoretic result 
was due to Hilbert~\cite{Hilbert} (much before 
Ramsey's~\cite{Ramsey}, which gave the name 
of the theory) and related to finite sums:

\begin{thm}[Hilbert, 1892]\label{hilbert}
For any finite partition of~$\mathbb N$ and 
any $n\in\mathbb N$ there exist a~part~$A$, 
a~sequence $(x_i)_{i<n}$ in~$\mathbb N$, and 
an infinite $B\subseteq\mathbb N$ such that 
$$
\bigcup_{b\in B}\bigl(b+\FS(x_i)_{i<n}\bigr)
\subseteq A.
$$
\end{thm}

Here $b+C$ denotes $\{b+c:c\in C\}$, the shift 
of the set~$C$ by~$b$.

Another early result, Schur's theorem~\cite{Schur}, 
involved only sums with two summands, but without shifts:

\begin{thm}[Schur, 1916]\label{schur}
For any finite partition of~$\mathbb N$ 
there exist a~part~$A$ and $x,y\in\mathbb N$ such that 
$$
\{x,y,x+y\}\subseteq A.
$$
\end{thm}

\begin{rmrk}
Interesting enough, Schur used this result to prove 
that $x^n+y^n\equiv z^n\pmod p$ has solutions 
for all sufficiently large prime~$p$.
\end{rmrk}

A~natural question is whether two summands $x,y$ 
can be improved to three summands $x,y,z$, or more. 
The answer is affirmative for any finite number; 
this theorem, though was explicitly written by Folkman 
and by Sanders in the late 1960s (and sometimes 
refered as to Folkman's theorem, 
see, e.g.,~\cite{Graham et al}), is actually 
an immediate consequence of Rado's much earlier 
theorem on systems of linear equations:

\begin{thm}[essentially Rado, 1933; 
Folkman, Sanders, 1969]
\label{folkman}
For any finite partition of~$\mathbb N$ and 
any $n\in\mathbb N$ there exist a~part~$A$ and 
an $n$-tuple $x_0,\dots,x_{n-1}$ such that 
$$
\FS(x_i)_{i<n}\subseteq A.
$$
\end{thm}

Clearly, then some part~$A$ should include 
even sets $\FS(x_i)_{i<n}$ for some (distinct) 
sequences $(x_i)_{i<n}$ for all $n\in\mathbb N$. 
Finally, one can ask whether it is possible to 
have a~{\it single\/} infinite sequence whose 
finite sums lie in some part. This was known as the 
Graham--Rothschild conjecture until it was proved by 
Hindman~\cite{Hindman 1974} (for the history behind 
its proof see Section~1 of~\cite{Hindman 2005}):

\begin{thm}[The Finite Sums Theorem, Hindman, 1974]%
\label{hindman}
For any finite partition of~$\mathbb N$ there exist 
a~part~$A$ and an infinite sequence $x_0,\dots,x_i,\ldots$ 
such that 
$$
\FS(x_i)_{i<\omega}\subseteq A.
$$
\end{thm}

Hindman's original proof in~\cite{Hindman 1974} 
was purely combinatorial and rather complicated.%
\footnote{In 2008 the author heard from Hindman 
an amusing phrase: ``I~never understood my proof''.} 
In the same year Baumgartner provided a~much shorter 
combinatorial proof~\cite{Baumgartner}. Although both 
proofs remain interesting (in particular, in point of 
view of reverse mathematics, see Remark~\ref{reverse}), 
only the third proof, based on the idea to use 
{\it idempotent ultrafilters\/}, was the beginning 
of a~new era in algebraic Ramsey theory.


\subsection{Idempotent ultrafilters}

The idea to use such type of ultrafilters for the proof 
of the Graham--Rothschild conjecture was proposed 
(but not published) by Galvin around~1970. 
In~\cite{Hindman 1972}, Hindman showed that under~CH 
(Continuum Hypothesis), the conjecture is equivalent 
to the existence of such ultrafilters. Finally, in~1975
Glazer observed that their existence follows from some 
topological facts, which have been already known; 
this proof was first published in~\cite{Comfort 1977}.

Recall that a~{\it groupoid\/} is a~set~$X$ with an 
arbitrary binary operation on it (e.g., a~semigroup is 
just an associative groupoid).%
\footnote{This is the terminology common in universal algebra, 
see, e.g., \cite{Burris Sankappanavar} or~\cite{Maltsev 1970}. 
In category theory and related fields, ``groupoid'' means 
a~partial group (i.e.~a~set with a~binary {\em partial} 
operation which is associative, invertible, and has an identity) 
or a~small category in which every morphism is invertible 
(which is essentially the same), see, e.g.,~\cite{Higgins}. 
A~groupoid in the sense of universal algebra is sometimes 
(following Bourbaki) called a~{\em magma}.}  
A~(multiplicatively written) groupoid $(X,\,\cdot\,)$ is 
{\it right topological\/} iff $X$~is endowed with a~topology 
in which all its right shifts, i.e.~the maps $x\mapsto xa$ 
for all $a\in X$, are continuous.

The first of the topological facts is the following 
statement, due in its final form to Ellis~\cite{Ellis}:

\begin{thm}[Ellis, 1969]\label{Ellis}
Every compact Hausdorff right topological semigroup 
has an idempotent. 
\end{thm}

The second fact is that $(\mathbb N,+)$, the additive 
semigroup of natural numbers, extends (in a~canonical way) 
to $(\scc\mathbb N,+)$, a~compact Hausdorff extremally 
disconnected semigroup of ultrafilters over~$\mathbb N$ 
which right topological; moreover, all its left shifts 
by principal ultrafilters are continuous.

Combining these two facts, we see that the semigroup 
$(\scc\mathbb N,+)$ has an idempotent. Moreover, it has 
a~{\it free\/} ultrafilter which is idempotent; to see, 
take rather $\mathbb N\setminus\{0\}$. 
It remains to use Galvin's observation:
If $u$~is an idempotent in  $(\scc\mathbb N,+)$, then 
for any $S\in u$ there is an infinite $(x_i)_{i<\omega}$ 
in~$\mathbb N$ such that $\FS(x_i)_{i<\omega}\subseteq S$. 
Now Hindman's Finite Sums Theorem (Theorem~\ref{hindman}) 
is immediate: whenever $u$~is an idempotent ultrafilter 
then one part of any finite partition does belong to~it.


A~real value of these observations, however, is that 
they have a~very broad character; in fact, they lead 
to a~general version of Theorem~\ref{hindman}, as was 
(according to~\cite{Comfort 1977}) independently 
pointed out by Glazer and Hindman around~1975.

Recall that $\scc X$, the set of ultrafilters over~$X$, 
carries the standard topology generated by open sets 
of form $\widetilde A=\{u\in\scc X:A\in u\}$ for all 
$A\subseteq X$, which is compact, Hausdorff, and 
extremally disconnected (the latter means that 
the closure of any open set is open).

Given a~groupoid $(X,\,\cdot\,)$, it canonically extends 
to the groupoid $(\scc X,\,\cdot\,)$ whose operation is 
defined by letting 
$
u\cdot v=
\{A\subseteq X:
\{x\in X:\{y\in X: 
x\cdot y\in A\}\in v\}\in u
\}
$
for all $u,v\in\scc X$, the {\it ultrafilter extension\/} 
of $(X,\,\cdot\,)$. (Here the word ``extension'' relates 
to the usual identification of elements of~$X$ with the 
principal ultrafilters over~$X$, under which one lets 
$X\subseteq\scc X$.)
Topologically the extension is going as follows:
first we continuously extend all left shifts of the 
initial operation, and then all right shifts of the 
obtained partially extended operation (see 
\cite{Saveliev} or~\cite{Saveliev (inftyproc)} for 
details and explaining of canonicity of the 
construction).

\begin{lmm}\label{beta X} 
Every discrete semigroup $(X,\,\cdot\,)$ canonically 
extends to the semigroup $(\scc X,\,\cdot\,)$. 
The latter is, w.r.t.~the standard compact Hausdorff 
extremally disconnected topology on~$\scc X$, a~right 
topological semigroup with continuous left shifts 
by principal ultrafilters.
\end{lmm}

A~groupoid is {\it weakly left cancellative\/} iff 
for any of its elements $a,b$ the equation $a\cdot x=b$ 
has only finitely many solutions. It can be verified 
that whenever $X$~is weakly left cancellative then 
the free ultrafilters form a~closed subgroupoid 
of $\scc X$; therefore, Ellis' theorem (Theorem~\ref{Ellis}) 
is applicable to weakly left cancellative semigroups. 
Thus combining \ref{Ellis} and~\ref{beta X}, we obtain:

\begin{coro}\label{free idempotent} 
For every semigroup $(X,\,\cdot\,)$,
the semigroup $(\scc X,\,\cdot\,)$ has an idempotent. 
Moreover, if $(X,\,\cdot\,)$ is infinite and either 
without idempotents or weakly left cancellative, 
then $(\scc X,\,\cdot\,)$ has a~free idempotent.
\end{coro}

The next result extends Galvin's observation 
to arbitrary semigroups:

\begin{lmm}\label{fs in idempotents} 
For every semigroup $(X,\,\cdot\,)$, if $u$~is 
a~free idempotent in $(\scc X,\,\cdot\,)$, then for 
any $S\in u$ there is an infinite $(x_i)_{i<\omega}$ 
in~$S$ such that $\FP(x_i)_{i<\omega}\subseteq S$. 
\end{lmm}

\begin{rmrk}\label{rmk: idempotent} 
1.
The statement converse to Lemma~\ref{fs in idempotents} 
is also true: any such~$S$ belongs to some idempotent 
ultrafilter. Thus ultrafilters each element of which 
includes some $\FP(x_i)_{i<\omega}$ are exactly those 
that belong to $\cl_{\scc X}\{u\in\scc X:u\cdot  u=u\}$,
the closure (in~$\scc X$) of the set of idempotents 
ultrafilters. 
(The latter set is not closed, see~\cite{Hindman Strauss}.)

Let us point out also that the existence of ultrafilters 
$u\in\scc\mathbb N$ such that any $S\in u$ includes a~set 
of form $\FS(x_i)_{i<\omega}$ which {\it itself\/} belongs 
to~$u$, called {\it strongly summable\/} ultrafilters, is 
independent of $\ZFC$ (the Zermelo--Fraenkel set theory); 
see~\cite{Hindman Strauss}, Chapter~12.

2. 
In fact, Lemma~\ref{fs in idempotents} holds 
for arbitrary groupoids, if we define for them 
$\FP(x_i)_{i}$ by putting parentheses in a~right way
(\cite{Saveliev (idempotents),Saveliev (Hindman)}).
\end{rmrk}


Now \ref{free idempotent} and \ref{fs in idempotents} 
together immediately give the general version of 
Hindman's theorem:

\begin{thm}[The Finite Products Theorem, 1975]%
\label{hindman general}
Let $(X,\cdot\,)$ be a~semigroup either without 
idempotents or weakly left cancellative.
For any finite partition of~$X$ there exist a~part~$A$ 
and an infinite sequence $x_0,\dots,x_i,\ldots$ such that 
$$
\FP(x_i)_{i<\omega}\subseteq A.
$$
\end{thm}

\begin{exmp}\label{multiplicative}
The multiplicative semigroup 
$(\mathbb N,\cdot\,)$ of natural numbers.
\end{exmp}


\subsection{Generalizations}

After these initial steps, algebra of ultrafilters 
quickly became an advanced area cultivated by many 
prominent authors (Bergelson, Blass, van Douwen, 
Hindman, Leader, Protasov, Strauss, Zelenyuk among 
others). 
Earlier results, including classical Ramsey's 
theorem and van der Waerden's arithmetic progressions 
theorem as well as newer Hales--Jewett's theorem and 
Furstenberg's multiple recurrence theorem, were 
restated by means of this technique. 
This provided a~better understanding of algebraic 
Ramsey theory and leaded to new deep results and 
applications to number theory, algebra, topological 
dynamics, and ergodic theory; most of them have 
no known elementary proofs.

Here we mention only two immediate generalizations of 
Hindman's theorem. The first provides a~simultaneous 
additive and multiplicative version of the theorem (see 
\cite{Hindman Strauss}, Corollary~5.22; the result was 
originally proved in~\cite{Hindman 1979}, an elementary 
proof can be found in~\cite{Bergelson Hindman}):

\begin{thm}[Hindman, 1979]%
\label{addit-multipl}
For any finite partition of~$\mathbb N$ there exist 
a~part~$A$ and two infinite sequences 
$x_0,\dots,x_i,\ldots$ and $y_0,\dots,y_i,\ldots$ 
such that 
$$
\FS(x_i)_{i<\omega}\cup\FP(y_i)_{i<\omega}\subseteq A.
$$
\end{thm}

The proof is not difficult modulo the above 
observations; it suffices first to show that the set 
$\cl_{\scc\mathbb N}(\{u\in\scc\mathbb N:u=u+u\})$ 
(the closure of the set of additive idempotents) 
forms a~left ideal of the multiplicative semigroup 
$(\scc\mathbb N,\,\cdot\,)$; then apply
Theorem~\ref{Ellis} to pick some $v=v\cdot v$ in~it.

\begin{rmrk}\label{one sequence non-exist}
This result cannot be improved by showing 
the existence of {\it a~single\/} sequence 
$(x_i)_{i<\omega}$ such that 
$\FS(x_i)_{i<\omega}\cup\FP(x_i)_{i<\omega}$ 
is included into a~part of a~given partition. 
In fact, this is impossible even for sums and 
products of all pairs of distinct elements in 
$(x_i)_{i<\omega}$\,; see~\cite{Hindman Strauss}, 
Theorem~17.16. 
\end{rmrk}


Another generalization of Hindman's theorem we 
want to formulate in this overview, is an its 
multidimensional (more precisely, finite-dimensional) 
version, also proved by Hindman in~\cite{Hindman 1979} 
(see also~\cite{Hindman Strauss}, Theorem~18.11;
an ``elementary'' proof was obtained 
in~\cite{Bergelson Hindman}):

\begin{thm}[Hindman, 1979]%
\label{multidimensional}
For any $m,n\in\mathbb N$, 
infinite semigroups $X_0,\ldots,X_n$ 
each of which is either without idempotents 
or weakly left cancellative, 
and finite partition of 
the Cartesian product $\prod_{i\le n}X_i$ 
there exist a~part~$A$, $m$-sequences $(x_{i,k})_{k<m}$ 
in each~$X_i$, $i<n$, and an infinite sequence 
$(x_{n,k})_{k<\omega}$ in~$X_n$ such that 
$$
\bigl(\prod_{i<n}\FP(x_{i,k})_{k<m}\bigr)
\times
\FP(x_{n,k})_{k<\omega}
\subseteq A.
$$
\end{thm}

\begin{rmrk}\label{two sequences non-exist}
This result cannot be improved by showing 
the existence of {\em two} infinite sequences, say, 
$(x_{n-1,k})_{k<\omega}$ and $(x_{n,k})_{k<\omega}$\,, 
generating such sets of finite products; 
see~\cite{Hindman Strauss}. 
\end{rmrk}

Theorem~\ref{multidimensional} will be crucial 
in our application to topologization of certain 
universal algebras discussed in Section~2.

For further information about the topic, we refer 
the reader to the literature. In particular, 
\cite{Hindman Strauss}~is a~comprehensive treatise 
on algebra of ultrafilters, with an historic 
information and a~vast list of references; 
\cite{Protasov}~provides a~clear introduction 
to this area; 
\cite{Comfort Negrepontis}~is the classical 
textbook on the general theory of ultrafilters.

\begin{rmrk}\label{reverse} 
There are investigations of Hindman's theorem 
and related assertions in point of view of their 
proof-theoretic and set-theoretic strength.

An arithmetic version of Hindman's theorem was 
studied in reverse mathematics. It was shown 
in~\cite{Blass et al} that its proof-teoretic strength 
(over $\RCA_0$) lies between $\ACA_0$ and $\ACA^{+}_0$. 
%
%
%
For more recent results see 
\cite{Towsner,Dzhafarov et al,Carlucci et al} 
and the literature mentioned there. 

Very recently it was shown~\cite{Tachtsis} that 
Ellis' theorem (Theorem~\ref{Ellis} here) follows 
from the Boolean Prime Ideal Theorem, thus showing 
it is weaker than the full~$\AC$ (the Axiom of Choice). 
The article~\cite{Tachtsis} contains also a~series 
of other relevant results. 
\end{rmrk}

\begin{rmrk}\label{ultraextensions}
In~\cite{Saveliev (idempotents)}, Theorems \ref{Ellis},
\ref{hindman general},~\ref{addit-multipl} were 
generalized to certain non-associative algebras. 
Further generalizations of Hindman's theorem involving 
larger sets of finite products (taken rather along 
partially ordered sets than sequences) were studied 
in~\cite{Saveliev (Hindman)}.

The above extension of semigroups by ultrafilters 
is a~partial case of a~certain canonical procedure 
of ultrafilter extension of arbitrary first-order 
models, which was defined 
in~\cite{Saveliev,Saveliev (inftyproc)}; 
some historical information can be found in 
Introduction of~\cite{Poliakov Saveliev 2021}. 
\end{rmrk}


\section{Finite sums and topologizations}

In this section, we apply the multidimensional 
generalization of the Finite Sums Theorem, which was 
formulated above (Theorem~\ref{multidimensional}), to 
the problem of topologizability of universal algebras. 

First we recall the origin of the problem and results 
obtained earlier; then we define a~certain class of 
universal algebras, called here ``polyrings'', which 
includes various classical algebras like rings, 
differential rings, algebras over rings, and others, 
and the Zariski topology of these algebras and their 
finite powers. 

After this, we formulate the main result of this 
section (Theorem~\ref{main thm}), which states that 
for any infinite polyring~$K$, any map of $K^n$ 
into~$K$ defined by a~term in $n$~variables is 
a~closed nowhere dense subset of~$K^{n+1}$ in its 
Zariski topology. In particular, $K^n$~is closed 
nowhere dense in~$K^{n+1}$, and a~fortiori, all 
the Zariski spaces~$K^n$ are non-discrete. Then 
we state that all countable polyrings are Hausdorff 
topologizable, and conclude the section with an 
outline of a~proof of Theorem~\ref{main thm}.

\subsection{Topologizations of algebras}

As usual, a~{\em universal algebra} is a~set with a~family 
of operations on it, i.e.~essentially a~first-order model 
in a~vocabulary consisting only of function symbols.%
\footnote{The term ``universal algebra'' (or sometimes 
``general algebra'') is also used for the area studying 
such algebras; in this meaning it apparently first appeared 
in the book~\cite{Whitehead} by Whitehead, where it was 
attributed to Sylvester.} We shall denote them by 
$(X,\Omega)$ or variants of this notation. We assume that 
any of the operations is of an arbitrary finite arity 
(although sometimes algebras with infinitary operations 
are found in the literature), where $0$-ary operations 
are constants. 
For more on universal algebras we refer the reader to 
the standard textbooks 
\cite{Burris Sankappanavar, Cohn, Gratzer, Maltsev 1970}.

Recall now some definitions and facts concerning 
the topologization problem.

\begin{dfn}\label{def: topologizability}
1. 
A~universal algebra is a~{\em topological algebra} 
iff it is endowed with a~topology in which all its 
main operations are continuous. In this case, the 
algebra and topology are said to be {\em compatible}.

2. 
A~universal algebra is $T_i$-{\it topologizable\/} 
iff it admits a~non-discrete $T_i$-topology which 
turns it into a~topological algebra.
\end{dfn}

Topological universal algebras were defined by 
Maltsev in~\cite{Maltsev}. A~topologizability 
problem was first posed by Markov\:(Jr.)~%
in~\cite{Markov 1944}--\cite{Markov 1946}, 
who asked whether any group admits a~non-discrete 
Hausdorff topology in which its multiplication and 
inversion become continuous. He (implicitly) defined 
a~$T_1$-topology on a~group, called now its Zariski 
topology, and proved that any countable group is 
$T_2$-topologizable iff its Zariski topology is 
non-discrete.

Later it was proved that the answer is 
affirmative for certain classes of groups, 
as Abelian~\cite{Kertesz Szele} 
and free~\cite{Zelenyuk 2011}, 
and negative in general; for uncountable groups 
this was proved under~$\CH$ in~\cite{Shelah} and 
without~$\CH$ in~\cite{Hesse}, for countable groups 
in~\cite{Olshanskii} (based on Adian's construction):

\begin{thm}\label{topologizability of groups}
$\phantom{.}$
\\
1~{\rm(Kert{\'e}sz and Szele, 1953, Zelenyuk, 2011)}. 
All Abelian groups as well as and all free groups 
are $T_2$-topologizable.
\\
2~{\rm(Shelah, 1976, Hesse, 1979, Olshanski, 1980)}. 
There exist groups of any infinite cardinality 
that are not $T_2$-topologizable. 
\end{thm}

There are known other classes of groups in which 
all countable groups are $T_2$-topologizable 
(see~\cite{Kotov}, Corollary~3). 
Other examples of non-topologizable groups 
(e.g.~non-topologizable torsion-free groups) 
can found in~\cite{Klyachko et al}. For further 
studies of topologizability of groups and other 
related problems posed by Markov 
in~\cite{Markov 1944}--\cite{Markov 1946} 
we refer the reader 
to~\cite{Banakh Protasov}--\cite{Dikranjan Toller}
and the literature there.

\begin{rmrk}
Zelenyuk proved~\cite{Zelenyuk} that any infinite 
group admits a~non-discrete zero-dimen\-sional 
$T_3$-topology in which all its left and right 
shifts and inversion are continuous. 
Cf.~also Section~9.2 in~\cite{Hindman Strauss},  
which discusses left topologizability of groups 
obtained by idempotent ultrafilters over them. 
\end{rmrk}


The situation with topologization of rings slightly 
differs from the case of groups.
It is still possible to prove that any countable 
ring is $T_2$-topologizable iff its Zariski topology 
is non-discrete. In~1970, Arnautov obtained the 
following principal results:
the Zariski topology of every infinite ring 
is non-discrete~\cite{Arnautov (countable rings)}
and so countable rings admit Hausdorff topologies,
the same holds for all commutative 
rings~\cite{Arnautov (commutative rings)}, but not 
in general~\cite{Arnautov (nontopologizable rings)}:

\begin{thm}[Arnautov, 1970]
\label{topologizability of rings}
$\phantom{.}$
\\
1.
All infinite rings have non-discrete 
Zariski topologies.
\\
2. 
All countable rings as well as and 
all commutative rings are $T_2$-topologizable.
\\
3.
There exist rings of an uncountable cardinality
that are not $T_2$-topologizable. 
\end{thm}

A~survey on topologizability of rings and modules 
can be found in~\cite{Arnautov et al}, Chapter~5. 


The topologizability problem was studied for other 
algebras. Prior the negative solution was obtained 
for groups, it was done for groupoids~\cite{Hanson} 
and semigroups~\cite{Taimanov semi}. For universal 
algebras this was studied in~\cite{Podewski}. That 
countable algebras admit Hausdorff topologizations 
iff their Zariski topologies are non-discrete, was 
shown for unoids (i.e.~algebras with arbitrary families 
of unary operations) in~\cite{Banakh Protasov Sipacheva}. 
In~\cite{Taimanov} this fact was announced 
for all universal algebras; the proof was provided  
in~\cite{Kotov} and completed in~\cite{Dutka Ivanov}. 
Moreover, the latter paper proves this fact for 
all first-order models (requiring that all 
their relations should be closed), and provides 
a~sufficient condition for $T_2$-topologizability 
of models of any cardinality; we shall use this 
condition below (Theorem~\ref{T3 polyring}). 

The $T_1$-topologizability generally does 
not imply the $T_2$-topologizability 
(see e.g.~\cite{Palyutin et al}). For certain varieties 
of algebras, however, any compatible 
topology satisfying a~weaker separability axiom 
in fact satisfies a~stronger one.  
The implication from $T_0$ to $T_1$ holds in a~variety 
iff it is congruence $n$-permutable for some~$n$ (the 
``if'' and the ``only if'' parts were proved in 
\cite{Gumm 1984} and \cite{Coleman 1997}, respectively). 
The implication from $T_0$ to $T_2$ holds in 
congruence permutable varieties \cite{Taylor 1977};  
this extends to congruence $3$-permutable
but not to $4$-permutable varieties
(\cite{Gumm 1984} and \cite{Coleman 1996}, respectively).
The congruence modularity suffices for an 
$n$-permutable variety to satisfy this implication 
\cite{Kearnes Sequeira}, and in certain subclasses 
of varieties it is also necessary~\cite{Bentz 2006}. 
The implication from $T_0$ to $T_3$ holds for 
groups~\cite{Montgomery Zippin} (we shall use this 
quasigroups~\cite{Chein et al}. A~characterization 
fact in Theorem~\ref{T3 polyring}) and, moreover, 
of this implication in a~special class of varieties 
is provided in~\cite{Bentz 2007}.


\subsection{Polyrings}

In~\cite{Protasov} Protasov gave an elegant proof 
of the non-discreteness of the Zariski topologies 
of rings by using Hindman's Finite Sums Theorem.
\footnote{
It is worth to note that the original proof 
in~\cite{Arnautov (nontopologizable rings)} used
van der Waerden's Arithmetic Progression Theorem, 
another standard result having a~short proof via 
ultrafilter algebra (which can be found 
in~\cite{Hindman Strauss} or~\cite{Protasov}). 
Likewise mentioned Folkman's theorem, 
van der Waerden's theorem is also an immediate 
consequence of Rado's theorem on systems of linear 
equations.
}
Following close ideas, but replacing this theorem 
with its stronger multidimensional version 
(Theorem~\ref{multidimensional}), we shall 
state stronger facts, from which will follow 
the non-discreteness of the Zariski topologies 
for universal algebras of a~much wider class, 
called here ``polyrings''. Moreover, it will 
follow that all finite powers of these algebras 
have non-discrete Zariski topologies, and even 
that for all finite~$n$, the $n$-th power is 
closed nowhere dense in the $(n+1)$-th power.

\begin{dfn}\label{def: polyring} 
A~universal algebra $(K,0,-,+,\Omega)$ is 
a~{\it polyring\/} iff $(K,0,-,+)$ is an Abelian group 
and any operation $F\in\Omega$ (of an arbitrary 
positive arity) is distributive w.r.t.\,the addition, 
i.e.~the additive shifts
$$
x\mapsto F(a_0,\ldots,a_{i-1},x,a_{i+1},\ldots,a_{n-1})
$$
are endomorphisms of the group $(K,0,-,+)$, 
for all \text{$i<n$} and 
$a_0,\ldots,a_{i-1},$ $a_{i+1},\ldots,a_{n-1}\in K$. 
(The case when $\Omega$ is empty is not excluded, 
so all Abelian groups are polyrings. On the other hand, 
we can assume that $\Omega$ contains constants for all 
points of~$K$, and in fact, we shall assume this in
Definition~\ref{def zariski}. The operation~$-$ is 
assumed to be unary, but we use the common notation 
$x-y$ meaning $x+(-y)$ when convenient.)
\end{dfn}

\begin{exmp}\label{exm: polyring}
Various classical algebras can be considered 
as special instances of polyrings: 
Abelian groups with operators, 
modules, Lie algebras, rings, 
differential rings, algebras over a~ring, 
Boolean algebras with operators 
(in sense of~\cite{Jonsson Tarski}), etc. 
\end{exmp}


Polyrings are not a~new class of universal algebras. 
Recall that {\em (multioperator) $\Omega$-groups} 
are groups, not necessarily Abelian, with operations 
$F\in\Omega$ of arbitrary arities satisfying 
$F(0,\ldots,0)=0$. This class was introduced 
in~\cite{Higgins 1956}, along with its special 
subclasses including distributive $\Omega$-groups 
(with $F\in\Omega$ distributive in the sense of 
Definition~\ref{def: polyring}) and, among them, 
those in which the additive group is Abelian.
The latters, known in the literature as 
``(multioperator) $\Omega$-rings'', are obviously 
what is called here by the shorter name ``polyrings''. 
Reviews on studies of this class, and on an even 
narrower class of so-called ``$\Omega$-algebras'', 
can be found in~\cite{Kurosh 1969,Baranovich Burgin}.


\begin{lmm}\label{lm: expansion}
There are two canonical expansion procedures:
\\
1. 
For every Abelian group $(K,0,-,+)$ there exists 
the {\it most expanded\/} polyring $(K,0,-,+,\Omega)$, 
i.e.~such that for any polyring $(K,0,-,+,\Omega')$ 
we have $\Omega'\subseteq\Omega$.
\\
2. 
For every universal algebra $(K',\Omega')$ 
there exists a~polyring $(K,0,-,+,\Omega)$ such that 
$(K',\Omega')$ is embedded into $(K,\Omega)$.
\end{lmm}

Clause~2 of Lemma~\ref{lm: expansion} generalizes 
the standard construction of a~groupoid ring, 
which provides the ring of formal sums with 
integer coefficients whose multiplicative groupoid 
extends a~given groupoid; in the general case for 
each $n$-ary operation $F'\in\Omega'$ we define 
the $n$-ary operation $F\in\Omega$ by using 
products of $n$~integers.


\subsection{Zariski topologies}

As said above, the Zariski topology on groups 
was implicitly considered by Markov in 
\cite{Markov 1944}--\cite{Markov 1946}. 
The first explicit description of it, under 
the name of ``verbal topology'', was given 
in~\cite{Bryant}. The name ``Zariski topology'' 
became standard after the paper~\cite{Baumslag et al}, 
in which the authors developed algebraic geometry 
over abstract groups; 
in analogy with classical algebraic geometry over 
fields they defined the Zariski topologies on finite
powers~$G^n$ of a~group~$G$ by using solution sets
of $n$~variables equations (the case $n=1$ gives 
the verbal topology of~\cite{Bryant}). 
Later this approach, together with close ideas 
in~\cite{Plotkin 1996,Plotkin 1997}
resulted in universal algebraic geometry; 
see~\cite{Daniyarova et al 2008,Daniyarova et al 2011} 
and the references there. The Zariski topologies 
of universal algebras were considered 
in~\cite{Podewski},~\cite{Taimanov}, later 
in~\cite{Kotov}, and generalized to arbitrary 
first-order models in~\cite{Dutka Ivanov}. 
We mention also~\cite{Zilber} providing an abstract,
model-theoretic approach to Zariski topologies of 
classical algebraic geometry.

For simplicity, here we define only the Zariski 
topologies of polyrings, although the general 
definition is not much more complicated (it 
involves the solutions of all atomic formulas).

Recall the standard definition of universal algebra (see 
e.g.~\cite{Maltsev 1970} or~\cite{Burris Sankappanavar}): 
a~{\em term} is an element of the smallest set that 
contains all variables and constants and closed under 
applying function symbols, i.e.~if $F$ is an $n$-ary 
function symbol and $F_0,\ldots,F_{n-1}$ are terms 
then $F(F_0,\ldots,F_{n-1})$ is a~term.%
\footnote{Sometimes terms are called ``words'' 
(e.g.~in~\cite{Higgins 1956}), which corresponds to 
the name ``verbal topology'', or else ``polynomial 
symbols'' (e.g.~in~\cite{Gratzer}), but we reserve 
here the name ``polynomials'' for terms of a~special 
kind, similar to ordinary polynomials in rings 
(Definition~\ref{dfn polynomial}).}
A~term is {\em closed} iff no variables occur in it, 
i.e.~iff it is constructed only from constants.

\begin{dfn}\label{def zariski}
Let $K$~be a~polyring and $n\in\mathbb N$.

1. 
If $F\in K[x_1,\ldots,x_n]$ is a~term 
in $n$~variables, let 
$$
S_F=
\bigl\{
(a_1,\ldots,a_n)\in K^n:F(a_1,\ldots,a_n)=0
\bigr\}
$$ 
denote the set of its {\it roots\/}, i.e.~solutions 
of the equation $F(x_1,\ldots,x_n)=0$ in~$K$. 
Finite unions of sets of form~$S_F$ 
are called {\it algebraic\/}.

2. 
A~set $S\subseteq K^n$ is {\it closed in 
the Zariski topology on\/}~$K^n$ iff $S$~is 
an intersection of algebraic sets.
In other words, sets of roots of equations 
in $n$~variables form a~closed subbase of the 
Zariski topology on~$K^n$, and the resulting 
algebraic sets a~closed base of the topology. 

3.
Usually, the Zariski topology is supposed to be 
a~$T_1$-topology, i.e.~each point is assumed to be 
definable by some term; obviously, it suffices to 
expand $\Omega$ by constants for all points. 
We follow this convention. 
\end{dfn}

So our Zariski topology on~$K$ 
is a~$T_1$- (but not necessarily~$T_2$-)topology, and 
in it, as easy to see, all shifts are continuous. 
Moreover:


\begin{lmm}\label{lm: zariski properties} 
Let $K$~be any polyring and $n\in\mathbb N$. Then:
\begin{itemize}
\setlength\itemsep{-0.39em}
\item[(i)] 
the Zariski topology on~$K^n$ is a~$T_1$-topology 
in which all maps of $K^n$ into~$K$ defined by 
terms in $n$~variables are continuous;
\item[(ii)] 
the Zariski topology on $K^{n+1}$ includes the 
product of the Zariski topologies on $K^n$ and~$K$, 
and can be strictly finer
(e.g.~this is so whenever $K$~is an infinite field);
\item[(iii)] 
the space~$K^n$ is homeomorphic to 
$K^n\!\times\!\{0\}\subseteq K^{n+1}$ 
(and will be identified with it below);
\item[(iv)] 
the space~$K^n$ is homogeneous. 
\end{itemize}
\end{lmm}

\begin{proof}
All items are immediate from 
Definition~\ref{def zariski}, and their proofs 
come down to verifications that certain sets 
have the form~$S_F$ for a~definable~$F$ and so 
is closed; nevertheless, we write them out.

(i). 
It suffices to check that, whenever $F:K^n\to K$ 
and $G:K\to K$ are definable by terms, then 
$F^{-1}S_G\subseteq K^n$ is closed in~$K^n$.  
Defining $H:K^n\to K$ by 
$H(x_1,\ldots,x_n):=G(F(x_1,\ldots,x_n))$, 
we get $F^{-1}S_G=S_H$.

(ii). 
Note first the general fact: for any family of 
spaces~$X_i$ and their fixed closed subbases $B_i$, 
all sets $\prod_iA_i$ with at most one index~$j$ 
such that $A_i\in B_i$ if $i=j$, and $A_i=X_i$ 
otherwise, form a~closed subbase of the space 
$\prod_iX_i$ endowed with the usual product topology. 
In our situation, it suffices to check that, 
whenever $F:K^n\to K$ and $G:K\to K$ are definable 
by terms, then the sets $S_F\times K$ and 
$K^n\times S_G$ are closed in the Zariski topology 
of~$K^{n+1}$. Defining $F',G':K^{n+1}\to K$ by 
$F'(x_1,\ldots,x_{n+1}):=F(x_1,\ldots,x_n)$ and 
$G'(x_1,\ldots,x_{n+1}):=G(x_{n+1})$, we get
$S_F\times K=S_{F'}$ and $K^n\times S_G=S_{G'}$.

(iii). 
As may expect, the natural injective map 
$h:K^n\to K^{n+1}$ defined by 
$h(x_1,\ldots,x_n):=(x_1,\ldots,x_n,0)$ is 
a~topological embedding. To show this, it 
suffices to check that, whenever $F:K^n\to K$ 
and $G:K^{n+1}\to K$ are definable by terms, 
then the sets $h[S_F]$ and $h^{-1}[S_G]$ 
(the image and preimage under~$h$) are closed 
in $K^n\!\times\!\{0\}$ (endowed with the induced 
topology as a~subspace of $K^{n+1}$) and $K^n$, 
respectively. Defining $F':K^{n+1}\to K$ by 
$F'(x_1,\ldots,x_{n+1}):=F(x_1,\ldots,x_n)$ 
and $G':K^n\to K$ by 
$G'(x_1,\ldots,x_n):=G(x_1,\ldots,x_n,0)$, 
we get $h[S_F]=S_{F'}\cap(K^n\!\times\!\{0\})$ 
and $h^{-1}[S_G]=S_{G'}$.

(iv). 
If $(a_1,\ldots,a_n),(b_1,\ldots,b_n)\in K^n$, 
letting 
$
h(x_1,\ldots,x_n):=
(x_1-a_1+b_1,\ldots,x_n-a_n+b_n),
$
we get an autohomeomorphism $h:K^n\to K^n$ with 
$h(a_1,\ldots,a_n)=(b_1,\ldots,b_n)$.  
\end{proof}


\begin{rmrk}
1. 
In Lemma~\ref{lm: zariski properties}, only 
clause~(iv) uses the group structure on~$K^n$ 
It is easy to see that (iv)~remains true for any 
structure with the transitive group of invertible 
shifts (cf.~Lemma~6 in~\cite{Kotov}).
Clauses (i)--(iii) are of a~general character. 

2. 
It is worth emphasizing that the Zariski topology 
of~$K$, even if it is non-discrete, generally
does not ensure its topologization (in the sense 
of Definition~\ref{def: topologizability}) since 
the operations of arity~$\ge2$ are not required to 
be continuous in it.
\end{rmrk}

For further information on Zariski topologies 
we refer the reader to the above literature.


\subsection{The main result}

Now we are ready to formulate the main result 
of this article (the proof will be provided 
in the next section)%
\footnote{
This theorem was established around~2010 but not 
published until~\cite{Saveliev (Zariski)} though 
presented at several conferences, e.g.~Colloquium 
Logicum~2012 in Paderborn, and included in author's 
course lectured at the Steklov Mathematical Institute 
in~2014.
}:

\begin{thm}[The Main Theorem]\label{main thm}
Let $K$~be an infinite polyring.
For any term $F\in K[x_1,\ldots,x_n]$
the mapping of $K^n$ into~$K$ defined by~$F$ 
is closed and nowhere dense in~$K^{n+1}$ 
(where $F$~is a~subspace of the space $K^{n+1}$ 
with its Zariski topology). In particular, 
$K^n$~is closed and nowhere dense in~$K^{n+1}$.
\end{thm}

Loosely speaking, this shows that such spaces, 
although can be even not Hausdorff, allow 
a~reasonable notion of topological dimension 
(this remark will be refined in Section~3.3). 
Certainly, this is much stronger fact than the 
non-discreteness of these space, which easily 
follows:

\begin{coro}\label{coro non-discrete}
If $K$~is an infinite polyring and $0<n<\omega$,
then $K^n$~with its Zariski topology has no 
isolated points.
\end{coro}

\begin{rmrk}\label{expanding polyring} 
The following observations are immediate from 
Lemma~\ref{lm: expansion}:

1. 
If $\Omega\subseteq\Omega'$ then the Zariski 
topology of $(K,0,-,+,\Omega')$ is finer than 
one of $(K,0,-,+,\Omega)$. 
Since there exists the most expanded polyring with 
a~given $(K,0,-,+)$, Theorem~\ref{main thm} provides 
the best possible result in this direction.

2. 
The procedure which embeds a~given universal algebra 
$(K',\Omega')$ into a~universal algebra $(K,\Omega)$ 
expanded to a~polyring $(K,0,-,+,\Omega)$ of formal 
sums, provides also a~natural extension of the Zariski 
space of any given algebra, which can be discrete, 
to a~non-discrete Zariski space of a~larger algebra. 
\end{rmrk}

Recall that for a~$T_1$-space~$X$, its
{\it pseudocharacter\/} at a~point $x\in X$ is 
the least cardinality of a~family of open sets 
whose intersection is~$\{x\}$, and the 
pseudocharacter of the whole space~$X$ is the 
supremum of these cardinals for all its points.
Clearly, the pseudocharacter of~$X$ does not 
exceed its cardinality~$|X|$, and is~$1$ 
if $X$~is discrete, and infinite otherwise.
Thus Corollary~\ref{coro non-discrete} shows 
that whenever $K$~is an infinite polyring endowed 
with its Zariski topology then its pseudocharacter 
is also infinite (and the same at all points 
due to Lemma~\ref{lm: zariski properties}(iv)).

\begin{thm}\label{T3 polyring}
Let $K$~be an infinite polyring, 
$|K|=\kappa\ge\omega$. 
\\
1. 
If $\kappa=\omega$, then $K$ is $T_3$-topologizable.
\\
2. 
If the pseudocharacter of $K$ with 
its Zariski topology is $\cf\kappa$, then 
$K$~is $T_3$-topologizable by some topology 
of the same pseudocharacter $\cf\kappa$. 
\end{thm}

\begin{proof}
By Theorem~\ref{main thm}, the Zariski topology 
of $K$ is non-discrete, so its pseudocharacter is 
infinite. Hence, if $K$~is countable then, 
by~\cite{Kotov,Dutka Ivanov}, it is topologizable 
by some Hausdorff topology without isolated points. 
In general, by~\cite{Dutka Ivanov}, if the 
pseudocharacter of $K$ (and in fact, of any universal 
algebra) in its Zariski topology is $\cf\kappa$, then 
$K$~is topologizable by some Hausdorff topology 
of the same pseudocharacter $\cf\kappa$. 
It remains to note that, by~\cite{Montgomery Zippin}, 
any $T_2$-topology compatible with a~group is 
actually~$T_3$ (operations in $\Omega$ play 
no role here). This completes the proof. 
\end{proof}

\begin{exmp}\label{exm: Arnautov}
1. 
Let $K$~be a~non-topologizable ring constructed by
Arnautov in~\cite{Arnautov (nontopologizable rings)}. 
In the Zariski space of~$K$, the pseudocharacter is 
infinite (as for any polyring) but not equal to $\cf|K|$. 

2. 
Let $K$~be $(\mathbb F_{2})^{\omega}$, the countable 
direct power of the two-element field~$\mathbb F_2$.
Clearly, $K$~is a~ring, and its Zariski topology 
coincides with the usual topology of the Cantor set. 

To see, note that for any $a=(a_i)_{i<\omega}\in K$, 
some $b=(b_i)_{i<\omega}\in K$ 
is a~solution of the equation $ax=0$ iff 
$b_i=0$ whenever $a_i=1$, 
and a~solution of the equation $ax+a=0$ iff
$b_i=1$ whenever $a_i=1$ 
(so we have $S_{ax+a}=K\setminus S_{ax}$).
It easily follows that any basic clopen set in 
the Cantor space is algebraic. 
\end{exmp}

\begin{rmrk}\label{not necessary} 
As $(\mathbb F_{2})^{\omega}$ has cardinality~$2^\omega$ 
and pseudocharacter $\omega<\cf 2^\omega$, we see that 
the sufficient condition from~\cite{Dutka Ivanov} is not 
necessary even for rings (although it is essential for 
the proof given there, and even for groups, 
cf.~\cite{Dutka Ivanov}, Proposition~3.4). 
\end{rmrk}


\subsection{The proof of the main result} 

Here we shall prove the Main Theorem 
(Theorem~\ref{main thm}). 
Throughout the section $K$ is a~polyring. 

\begin{ntt}\label{notation polyring}
To simplify the reading, we adopt the following conventions:
\begin{itemize}
\setlength\itemsep{-0.39em}
\item[(i)] 
Given $n<\omega$, we write $\vec x$, $\vec 0$, $\vec K$ 
instead of $(x_1,\ldots,x_{n})$, $(0,\ldots,0)$, $K^n$, 
respectively. 
\item[(ii)] 
The pointwise addition in~$\vec K$ is denoted also by~$+$, 
so $\vec x+\vec y$ denotes $(x_1+y_1,\ldots,x_{n}+y_{n})$. 
\item[(iii)] 
$K[x_1,\ldots,x_n]$, shortly $K[\vec x]$, denotes the set 
of all terms over~$K$ in the variables $x_1,\ldots,x_n$.
\end{itemize}
\end{ntt}

Let us now specify certain terms in polyrings 
(keeping the terminology of~\cite{Higgins 1956}):

\begin{dfn}\label{dfn polynomial}
1. 
A~term is a~{\it monomial\/} iff it does not 
contain~$+$\,, and a~{\it polynomial\/} iff it is 
either a~monomial or a~sum of (two or more) monomials. 

2. 
If $F$ and~$G$ are terms and $A$ a~class of polyrings,  
$G$~{\em represents $F$ in~$A$} (and conversely) iff 
$F(\vec a)=G(\vec a)$ for any $K\in A$ and $\vec a\in K$. 
\end{dfn}


\begin{lmm}\label{lmm polynomial}
In polyrings, any term is represented by a~polynomial. 
Moreover, if $F\in\Omega$ is an $(n+1)$-ary operation, 
then 
$$
F\bigl(
\sum_{j<k_0}\!x_{0,j},\ldots,\sum_{j<k_n}\!x_{n,j}
\bigr)
=
\sum_{j_0<k_0}\ldots\sum_{j_n<k_n}\!
F(x_{0,j_0},\ldots,x_{n,j_n}).
$$
\end{lmm}

\begin{proof}
Renaming $\sum_{j<k_i}x_{i,j}$ by~$y_i$, 
we first obtain
$
F(y_{0},\ldots,y_{n})=
\sum_{j_0<k_0}
F(x_{0,j_0},y_{1},\ldots,y_{n})
$
by using distributivity $k_0$~times, 
then similarly
$
F(x_{0,j_0},y_{1},\ldots,y_{n})=
\sum_{j_1<k_1}
F(x_{0,j_0},x_{1,j_1},y_{2},\ldots,y_{n}),
$
etc.~up to 
$
F(x_{0,j_0},x_{1,j_1},\ldots,y_{n})=
\sum_{j_n<k_n}
F(x_{0,j_0},x_{1,j_1},\ldots,x_{n,j_n}).
$
\end{proof}


\begin{dfn}\label{degree}
If $F$ is a~term over a~polyring~$K$ and $X$~a~set 
of variables, the {\em degree of~$F$ w.r.t.~$X$} 
is defined by induction on the construction of~$F$: 
it is equal to
\begin{itemize}
\setlength\itemsep{-0.39em}
\item[(i)] 
$0$ if $F$ is either a~constant or a~variable not in~$X$;
\item[(ii)] 
$1$ if $F$ is a~variable in the set;
\item[(iii)] 
the maximum of the degrees w.r.t.~$X$ of $F_0$ and $F_1$ 
if $F$ is $F_0+F_1$;
\item[(iv)] 
the sum of the degrees w.r.t.~$X$ of $F_0,\ldots,F_n$ 
if $F$ is $G(F_0,\ldots,F_n)$ for some $G\in\Omega$. 
\end{itemize}
The {\em degree of~$F$} is 
its degree w.r.t.~all variables.
\end{dfn}

\hide{ 
\begin{dfn}\label{degree1}
If $F$ is a~term over a~polyring~$K$ and $X$~a~set 
of variables, the {\em degree of~$F$ w.r.t.~$X$}, 
denoted by $\deg_X(F)$, is defined by induction on 
the construction of~$F$, letting it equal to:
\begin{itemize}
\setlength\itemsep{-0.39em}
\item[(i)] 
$0$ if $F$ is either a~constant or a~variable not in~$X$;
\item[(ii)] 
$1$ if $F$ is a~variable in~$X$;
\item[(iii)] 
$\max(\deg_X(F_0),\deg_X(F_1))$
if $F$ is $F_0+F_1$;
\item[(iv)] 
$\deg_X(F_0)+\ldots+\deg_X(F_n)$ 
if $F$ is $G(F_0,\ldots,F_n)$ for some $G\in\Omega$. 
\end{itemize}
The {\em degree of~$F$}, denoted by $\deg(F)$, is 
its degree w.r.t.~all variables.
\end{dfn}

It is easily follows that 
if $F$~is a~closed term then $\deg_X(F)=0$, and 
if $F$~is a~monomial then $\deg_X(F)$ is equal to
the number of occurrences in~$F$ of those variables 
that are in~$X$.  
}

It is easily follows that 
if $F$~is a~closed term then its degree is~$0$, and 
if $F$~is a~monomial then its degree w.r.t.~$X$ is 
equal to the number of occurrences in~$F$ of those 
variables that are in~$X$.

\begin{rmrk}
This definition of degrees is purely syntactic 
(e.g., according to it, the degree of $x^2-x^2$ is~$2$ 
and not~$0$) 
as if $K$~was a~free algebra in the vocabulary 
$\{0,-,+,\Omega\}$. 
We could define, for a~given class~$A$ of 
polyrings, the {\em true degree of $F$ w.r.t.~$X$ in~$A$} 
as the minimum of the degrees w.r.t.~$X$ of terms that 
represent~$F$ in~$A$. For the purpose of the article,
however, we do not need this invariant. 
\end{rmrk}


The following lemma states that in polyrings, 
any mapping defined by a~term is ``almost an 
endomorphism'', namely, an endomorphism up to 
a~mapping defined by a~term of a~lesser degree:

\begin{lmm}\label{lmm: almost endo}
Let $\vec x$ and $\vec y$ have the same length. 
For every $F\in K[\vec x]$ of a~nonzero degree there exists 
$G\in K[\vec x,\vec y]$ of a~lesser degree w.r.t.~$\vec x$ 
such that 
$$
F(\vec x+\vec y)=
F(\vec x)+F(\vec y)+G(\vec x,\vec y).
$$
\end{lmm}

\begin{proof}
This is a~routine check by induction on the construction of 
the term~$F$. We write out the details for the skeptic reader.

Let the term $F(\vec x)$ be of the form 
$F_0(\vec x)+F_1(\vec x)$ (we do not assume that all 
the variables of~$\vec x$ actually occur in both of 
$F_0$ and~$F_1$), and suppose that we have already proved, 
for $i\in\{0,1\}$, the identities
$
F_i(\vec x+\vec y)=
F_i(\vec x)+F_i(\vec y)+G_i(\vec x,\vec y)
$
with some $G_i$ of a~lesser degree w.r.t.~$\vec x$ 
than~$F_i$. We get (by using that the addition of 
polyrings is commutative and associative):
\begin{align*}
F(\vec x+\vec y)
&=
F_0(\vec x+\vec y)+F_1(\vec x+\vec y)
\\
&=
F_0(\vec x)+F_0(\vec y)+G_0(\vec x,\vec y)+
F_1(\vec x)+F_1(\vec y)+G_1(\vec x,\vec y)
\\
&=
F_0(\vec x)+F_1(\vec x)+F_0(\vec y)+F_1(\vec y)+
G_0(\vec x,\vec y)+G_1(\vec x,\vec y)=
F(\vec x)+F(\vec y)+G(\vec x,\vec y),
\end{align*}
where 
$G(\vec x,\vec y):=G_0(\vec x,\vec y)+G_1(\vec x,\vec y)$.
The degree w.r.t.~$\vec x$ of $G$ is equal to the maximum 
of such degrees of $G_0$ and~$G_1$, and so is less than 
the degree w.r.t.~$\vec x$ of $F$, which is equal to the 
maximum of such degrees of $F_0$ and~$F_1$.

Let now the term $F(\vec x)$ be of the form 
$H(F_0(\vec x),\ldots,F_n(\vec x))$ for some $H\in\Omega$
(again we do not assume that all the variables of~$\vec x$ 
occur in each of $F_0,\ldots,F_n$), and suppose that 
we have already proved, for $i\le n$, the identities
$
F_i(\vec x+\vec y)=
F_i(\vec x)+F_i(\vec y)+G_i(\vec x,\vec y)
$
with some $G_i$ of a~lesser degree w.r.t.~$\vec x$ 
than~$F_i$. By using Lemma~\ref{lmm polynomial}, we get:
\begin{align*}
F(\vec x+\vec y)
&=
H(F_0(\vec x+\vec y),\ldots,F_n(\vec x+\vec y))
\\
&=
H(F_0(\vec x)+F_0(\vec y)+G_0(\vec x,\vec y),
\ldots,
F_n(\vec x)+F_n(\vec y)+G_n(\vec x,\vec y))
\\
&=
\sum_{j_0<3}\ldots\sum_{j_n<3}\!
H(F_{0,j_0}(\vec x,\vec y),
\ldots,F_{n,j_n}(\vec x,\vec y)),
\end{align*}
where 
\hide{
$F_{i,j}(\vec x,\vec y))$ denotes 
$F_i(\vec x)$ if $j=0$,
$F_i(\vec y)$ if $j=1$, and 
$G_i(\vec x,\vec y)$ if $j=2$.
}
\begin{align*}
F_{i,j}(\vec x,\vec y)):=
\begin{cases}
F_i(\vec x)
&\text{if }j=0,
\\
F_i(\vec y)
&\text{if }j=1,
\\
G_i(\vec x,\vec y)
&\text{if }j=2.
\end{cases}
\end{align*}

It follows from the induction hypothesis that, 
among all summands of the form 
$$
H(F_{0,j_0}(\vec x,\vec y),
\ldots,F_{n,j_n}(\vec x,\vec y)),
$$
the term of the highest degree w.r.t.~$\vec x$ is 
the one in which $j_0=\ldots=j_n=0$, i.e., the term 
$H(F_{0}(\vec x),\ldots,F_{n}(\vec x))$, 
which is just $F(\vec x)$. So the degree w.r.t.~$\vec x$ 
of $F(\vec x+\vec y)$ is equal to one of $F(\vec x)$
and greater than one of the sum of all other summands.
Likewise, the summand with $j_0=\ldots=j_n=1$, i.e., 
the term $H(F_{0}(\vec y),\ldots,F_{n}(\vec y))$, is 
just $F(\vec y)$ (and its degree w.r.t.~$\vec y$ is 
equal to one of $F(\vec x+\vec y)$ and is greater than 
one of the sum of all other summands). 

Therefore, letting $G(\vec x,\vec y)$ to be the sum above 
without the two summands $F(\vec x)$ and $F(\vec y)$, 
we see that $G$ has the lesser degree w.r.t.~$\vec x$ 
(and w.r.t.~$\vec y$) than 
$F(\vec x+\vec y)=F(\vec x)+F(\vec y)+G(\vec x,\vec y)$,
as required. 
\end{proof}

\begin{rmrk}\label{Abelian}
Lemma~\ref{lmm: almost endo} 
(more precisely, that part of its proof that 
passes from $F_0(\vec x)$ and $F_1(\vec x)$ to 
$F_0(\vec x)+F_1(\vec x)$) is the only place in the 
whole proof of the Main Theorem using the assumption 
that the groups $(K,0,-,+)$ in polyrings are Abelian. 
Recall, however, that the Zariski topology of 
non-Abelian groups can be discrete; therefore, 
the assumption is essential. On the other hand, 
if an $\Omega$-group satisfies this lemma 
(or rather, its Corollary~\ref{coro: almost endo})  
then it satisfies the Main Theorem as well.
\end{rmrk}

\begin{coro}\label{coro: almost endo}
For every $F\in K[\vec x]$ of a~nonzero degree 
and any $\vec a\in\vec K$ 
there exists $H\in K[\vec x]$ of a~lesser degree 
such that 
$$
F(\vec x+\vec a)=F(\vec x)+H(\vec x).
$$
\end{coro}


The next lemma connects roots of terms and sets 
of finite sums; this crucial fact will allow us 
to apply Hindman-type results to our purposes:

\begin{lmm}[The Key Lemma]\label{key lmm}
Let $F\in K[\vec x]$ have the degree~$\le n$ and 
$(\vec a_i)_{i\le n}\in\vec K^{\,n+1}$. 
If $F(\vec b)=0$ for all $\vec b\in\FS(\vec a_i)_{i\le n}$, 
then $F(\vec 0)=0$.
\end{lmm}

\begin{proof}
We argue by induction on~$n$ using 
Corollary~\ref{coro: almost endo}.

Let $n=0$. So the degree of $F$ is~$0$, which means 
that $F$ is constant. Therefore, as $F(\vec a_0)=0$ by 
the condition, we conclude $F(\vec 0)=0$. 

Let $n>0$ and assume we have already proved the assertion 
for all $m<n$. Let $\vec a\in\FS(\vec a_i)_{i<n}$, so 
$\{\vec a,\vec a+\vec a_n\}\subseteq\FS(\vec a_i)_{i\le n}$, 
and hence $F(\vec a)=F(\vec a+\vec a_n)=0$ by the condition. 
By Corollary~\ref{coro: almost endo}, we also have 
$F(\vec a+\vec a_n)=F(\vec a)+H(\vec a)$ for some 
$H\in K[\vec x]$ of the degree~$<n$. Together this gives 
us $0=0+H(\vec a)$, i.e., $H(\vec a)=0$. Thus we have 
$H(\vec a)=0$ for all $\vec a\in\FS(\vec a_i)_{i<n}$, 
where the degree of $H$ is~$<n$, whence $H(\vec 0)=0$ 
follows by the induction hypothesis. Finally, 
again by Corollary~\ref{coro: almost endo}, we have 
$F(\vec a)=F(\vec 0+\vec a)=F(\vec 0)+H(\vec 0)$, which 
gives us $0=F(\vec 0)+0$, i.e., $F(\vec 0)=0$, as required. 
\end{proof}

Thus $\FS(\vec a_i)_{i\le n}\subseteq S_F$ implies 
$\vec 0\in S_F$ whenever the degree of $F$ is~$\le n$.
In terms of the Zariski topology, we obtain:

\begin{coro}\label{coro: fs in zariski}
For any $(\vec a_i)_{i<\omega}$ in $\vec K$,
the closure of the set $\FS(\vec a_i)_{i<\omega}$ 
in the Zariski topology of~$\vec K$ 
has the element~$\vec 0$.
\end{coro}

\begin{rmrk}\label{rmk: idempotent over polyring}
Recall that sets $S\supseteq\FS(\vec a_i)_{i<\omega}$ 
are exactly elements of idempotent ultrafilters in 
the semigroup $(\scc K,+)$ (see \cite{Hindman Strauss}, 
Theorem~5.12; also note the comment below the theorem);
cf.~Remark~\ref{rmk: idempotent}. 
\end{rmrk}


\begin{dfn}\label{dfn: finite-valued} 
A~set $A\subseteq\prod_{i<n}X_i$ is  
{\it finite-valued\/} iff there is $j<n$ such that 
all sections in~$X_j$ are finite, i.e.~for all 
$(a_i)_{i<j}\in\prod_{i<j}X_i$ and 
$(a_i)_{j<i<n}\in\prod_{j<i<n}X_i$ 
the sets
$$
\{a_j\in X_j:(a_i)_{i<n}\in A\}
$$
are finite. 
(Such an~$A$ can be regarded as a~partial finite-valued 
map of $\prod_{i\in n\setminus\{j\}}X_i$ into~$X_j$\,, 
which explains the name.)
\end{dfn}

The following theorem immediately leads to the main result 
but is also interesting in its own right:

\begin{thm}[The Key Theorem]\label{key thm}
For every infinite polyring~$K$, if $A\subseteq\vec K$ 
is finite-valued then $A$~has the empty interior 
in the Zariski topology of~$\vec K$.
\end{thm}

\begin{proof}
Notice first that if a~set~$A$ is finite-valued 
then so is any $B\subseteq A$. Therefore, it suffices 
to show that $A$~is not open.

Furthermore, notice that we can w.l.g.~suppose that 
$\vec 0\in A$ (otherwise additively shift the set 
by using that the space is homogeneous) and, 
assuming $\vec K$~denotes~$K^{n+1}$, that 
$A$~is finite-valued 
in the $n$th coordinate (otherwise rename the coordinates).

Toward a~contradiction, assume that $A$~is open. 
Then $\vec K\setminus A$ is closed so is the intersection 
of sets of form $S_{F_0}\cup\ldots\cup S_{F_j}$ 
for some terms $F_0,\ldots,F_j$. 
Pick any of these sets and show that it has~$\vec 0$. 
It will follow that $\vec 0\in\vec K\setminus A$, 
thus reaching a~contradiction.

Let $m$~be maximum of the degrees of the terms 
$F_0,\ldots,F_j$. The sets $A,S_{F_0},\ldots,S_{F_j}$
cover the whole space~$\vec K=K^{n+1}$:
$$
A\cup S_{F_0}\cup\ldots\cup S_{F_j}
=K^{n+1}.
$$
Therefore, by Theorem~\ref{multidimensional} 
(the multidimensional generalization of the Finite 
Sums Theorem), some of these sets includes the set 
$$
P=
\bigl(\prod_{i<n}\FS(a_{i,k})_{k\le m}\bigr)
\times
\FS(a_{n,k})_{k<\omega}
$$ 
for some sequences $(a_{i,k})_{k\le m}$ and 
$(a_{n,k})_{k<\omega}$ consisting of distinct 
elements of~$K$.

For~$A$, however, this is impossible: $P\nsubseteq A$ 
since $A$~is finite-valued in the $n$th coordinate. 
Hence, $P\subseteq S_F$ for some $F\in\{F_0,\ldots,F_j\}$. 
But then $\vec 0\in S_F$ follows from the Key Lemma 
(Lemma~\ref{key lmm}).

This completes the proof of Theorem~\ref{key thm}.
\end{proof}


\begin{coro}\label{coro of key thm}
For every infinite polyring~$K$, if $A\subseteq\vec K$ 
is finite-valued and closed in the Zariski topology 
of~$\vec K$, then $A$~is nowhere dense.
\end{coro}

\begin{proof}
For closed sets, to have empty interior is the same 
that to be nowhere dense. 
\end{proof}

\begin{rmrk}\label{rmk: non-closed finite-valued}
The assumption that $A$~is closed cannot be 
omitted even for single-valued maps; e.g., the set 
$A=\{(a,a^2):a\in\mathbb Z\}$, where $\mathbb Z$~is 
the additive group of integers, is everywhere dense 
in~$\mathbb Z^2$. 
\end{rmrk}

Now Theorem~\ref{main thm} follows since 
$F\subseteq K^{n+1}$ is single-valued and closed 
in~$K^{n+1}$ as $F=S_G$ for the map~$G$ defined 
by letting
$$
G(x_0,\ldots,x_{n-1},x_n)=F(x_0,\ldots,x_{n-1})-x_n.
$$
The proof of the Main Theorem is complete.


\section{Problems}

In this section, we provide several problems 
and tasks related to subjects discussed above.%
\footnote{Let us point out also that Section~8
of~\cite{Dikranjan Toller} and Section~13 
in~\cite{Zelenyuk 2011} both provide lists of 
open problems concerning topologies on groups 
and related questions.}
Unless otherwise stated, all algebras below 
are considered with their Zariski topologies 
(e.g.~when we ask whether a~group is connected, 
we actually ask whether so is its Zariski topology).

\subsection{Discreteness}

\begin{prbl}\label{prb: characterize non-discrete}
Characterize groups whose Zariski topology is non-discrete. 
Is the class of such groups first-order axiomatizable?
at least, second-order axiomatizable?

The same questions for other algebraic structures.

%
\end{prbl}

\begin{prbl}\label{prb: characterize Hausdorff}
Is the class of $T_2$-topologizable groups 
first-order (or at least, second-order) 
axiomatizable?

The same questions for other algebraic structures.

%
\end{prbl}

\begin{prbl}\label{prb: adding automorphism to non-discrete}
If the Zariski topology of a~group is non-discrete, 
does it remain non-discrete under adding an automorphism
(endomorphism) as a~new operation? all under adding all 
automorphisms (endomorphisms)?

The same questions for other algebraic structures.

%
\end{prbl}

\begin{prbl}\label{prb: non-discrete square}
If the Zariski topology of a~group~$K$ is non-discrete, 
is so the Zariski topology of~$K^2$? of~$K^n$? 
Furthermore, is then $K$ nowhere dense in~$K^2$? 
$K^n$~nowhere dense in~$K^{n+1}$ for all $n\in\mathbb N$? 

The same questions for skew (i.e.~non-commutative) fields, 
rings, polyrings, other algebraic structures.

\end{prbl}

Notice that Problems 
\ref{prb: characterize non-discrete}--%
\ref{prb: non-discrete square} are open only 
for non-Abelian groups; for Abelian case all four 
answers are affirmative by results of Section~2.

Let us say a~polyring $(K,0,-,+,\Omega)$ is 
{\it commutative\/} iff for all $F\in\Omega$ 
and permutations~$\sigma$ of~$n$, where $n$~is 
the arity of~$F$, we have 
$F(x_1,\ldots,x_n)=F(x_{\sigma1},\ldots,x_{\sigma n})$
for all $x_1,\ldots,x_n$ in~$K$.

\begin{prbl}\label{prb: comm polyring}
Are all commutative polyrings $T_2$-topologizable?
\end{prbl}

By~\cite{Arnautov (commutative rings)}, 
the answer are affirmative for commutative rings 
(see Theorem~\ref{topologizability of rings}).

\begin{prbl}\label{prb: discrete free}
Is there a~variety (or at least, a~quasivariety) 
of universal algebras, e.g.~groupoids, having 
an infinite free algebra whose Zariski topology 
is discrete? 
\end{prbl}


\subsection{Connectedness}

\begin{prbl}\label{prb: characterize connected}
Characterize groups whose Zariski topology is connected. 
Is the class of such groups first-order axiomatizable?
second-order axiomatizable?

The same questions for other algebraic structures
(Abelian groups, rings, skew fields, etc.).

\end{prbl}

\begin{prbl}\label{prb: adding automorphism to connected}
If the Zariski topology of a~group is connected, 
does it remain connected under adding an automorphism
(endomorphism) as a~new operation? all under adding all 
automorphisms (endomorphisms)?

The same questions for other algebraic structures.

%
\end{prbl}


\hide{

\begin{prbl}\label{prb: torsion-free connected} 
Are all Abelian torsion-free groups connected?

\end{prbl}

Recall that \cite{Klyachko et al}~provides 
examples of (non-Abelian) torsion-free groups
with discrete Zariski topologies.

[Dikran Dikranjan's comments: 

Do you mean “connected” in their Zariski topology? 
If “yes”, then also the answer is “yes”. 
Actually, it is enough the ask non-torsion. 
The precise result is even more striking: 
an Abelian group~$G$ is disconnected
in its Zariski topology if and only if 
there exists a~positive integer~$n$ such that
$nG$ is finite but non-trivial. 
It means, of course, that $G$ is a~bounded torsion group
(= of finite exponent), but not only. 
Such a~group is a~direct sum of cyclic groups, so 
completely determined by its Ulm--Kaplanski invariants 
$f_{p,k}$ that tell you how many times a~given cyclic 
$p$-group $\mathbb Z(p^k)$ appears in that direct sum. 
Call a~non-zero U.-K.-invariant {\it leading\/} 
if its $k$ is a~maximum (for a~fixed~$p$). 
This finitely many leading U.-K.’s tell you 
everything about the connectedness of $G$ 
in its Zariski topology~-- it is disconnected 
precisely [if] all of them are finite. 

I think this answers your question. This answer is 
somewhat hidden (but exists :-) in our 2010 paper [35]. 
More recently we write the paper that explains this 
in a~more explicit way by proving 
this curious theorem (especially from the point of 
view of your remarkable paper)~-— an Abelian group 
admits a~connected Hausdorff group topology iff 
its Zariski topology is connected! 
(This resolves an open problem of Markov on 
the connected topologization of groups.) 
This aspect (and more) you can find in another recent 
survey of mine (a~continuation of [36]) that 
my collaborator will send you later today.]

\begin{prbl}\label{prb: Abelian disconnected}
If an [infinite] Abelian group is disconnected, 
can it be decomposed in a~direct product with 
a~finite (of size~$>1$) factor? What about rings? 

\end{prbl}

[Dikran Dikranjan's comments: 

As far as Abelian groups are concerned, the answer 
is yes, if I correctly understand it and if $G$ is 
INFINITE. Namely, let us start with the finite case. 
If G is finite, but non-trivial, then it is 
disconnected, yet not always decomposable 
(e.g., when it is a~non-trivial cyclic p-group). 
If an infinite $G$ has disconnected, then, as 
explained above, it will be an infinite direct sum 
of finite (cyclic) groups. 
Is this the answer you expect?
I am afraid I do not understand your question.] 

}


Recall that \cite{Klyachko et al}~provides examples 
of non-Abelian torsion-free groups with the discrete 
Zariski topologies. The case of Abelian groups is 
opposite; it follows from the results of Dikranjan 
and Shakhmatov \cite{Dikranjan Shakhmatov} that 
all Abelian torsion-free and even non-torsion 
groups are connected in their Zariski topology. 
Moreover, they proved that the Zariski topology of 
an Abelian group~$G$ is disconnected iff there exists 
a~positive $n<\omega$ such that $nG$ is finite but 
non-trivial; this follows from their characterization 
of the connectedness of~$G$ in terms of Ulm--Kaplanski 
invariants. 
Another remarkable result was established in their 
subsequent paper \cite{Dikranjan Shakhmatov 2016}: 
an Abelian group admits a~connected Hausdorff group 
topology iff its Zariski topology is connected. 
This resolved a~long-standing open problem of Markov 
on the connected topologization of groups.


\begin{prbl}\label{prb: free group connected}
Is any non-Abelian free group connected? 
The same question for algebraically closed groups 
or, wider, infinite simple groups.

\end{prbl}

\begin{prbl}\label{prb: connected rings}
Is a~ring connected if its additive group is 
free? algebraically closed?

\end{prbl}

\begin{prbl}\label{prb: zero divisor disconnected}
Is any ring with non-trivial zero divisors 
disconnected?

\end{prbl}

\begin{prbl}\label{prb: module disconnected}
Is the free $R$-module connected whenever 
so is the ring~$R$?

\end{prbl}

\begin{prbl}\label{prb: skew field connected}
Is any skew field connected? If yes, what about 
quasifields (near-fields, rings without non-trivial 
zero divisors, etc.)?

\end{prbl}

The connectedness may fail for some rings. E.g.~the 
Zariski topology of $(\mathbb Z^\omega,+,\,\cdot\,)$ 
coincides with the usual topology on the Cantor set, 
and thus is disconnected. Its subring 
$(\mathbb Z^{<\omega},+,\,\cdot\,)$ consisting of all 
eventually zero sequences is homeomorphic to the space 
of rationals.

\begin{prbl}\label{prb: non-roots} 
Let $K$~be an infinite skew field and $F\in K[x]$ 
a~polynomial such that not all elements of~$K$ are 
roots of~$F$. Is there a~sequence $(a_i)_{i\in\mathbb N}$ 
of distinct elements of~$K$ such that all the elements of 
the set $\FS(a_i)_{i\in\mathbb N}$ are not roots of~$F$?

What about near-fields
(rings without non-trivial zero divisors, etc.)?

\end{prbl}

Let us point out (without a~proof) that an affirmative answer 
to Problem~\ref{prb: non-roots} implies an affirmative 
answer to Problem~\ref{prb: skew field connected}.


\subsection{Dimension}

Let $\ind(X)\ge-1$ for all topological spaces~$X$, 
and let: $\ind(X)=-1$ iff $X$~is empty, and whenever
spaces~$Y$ with $\ind(Y)\le n$ have already been 
defined then $\ind(X)\le n+1$ iff there exists 
an open base~$\varGamma$ of~$X$ such that
$\ind(\partial O)\le n$ for all $O\in\varGamma$. 
(Here $\partial S$~denotes the boundary of the set~$S$.)


Let us mention (without a~proof) that 
Theorem~\ref{main thm} gives the inequality 
$\ind(K^n)\ge(\ind K)+n-1$ for all polyrings~$K$ 
and $n\in\mathbb N$.

\begin{prbl}\label{prb: ind for field}
Is it true that for any infinite field~$K$ and all
$n\in\mathbb N$ we have $\ind(K^n)=n$? 
(It is not difficult to verify that this is the case 
for $n\le2$, and that $\ind(K^n)\ge n$ for $n>1$.) 

\end{prbl}

Let us point out (without a~proof) that an affirmative 
answer to Problem~\ref{prb: roots fields loc homeom} 
below implies an affirmative answer to 
Problem~\ref{prb: ind for field}.


\begin{prbl}\label{prb: ind quaternion}
Calculate $\ind(\mathbb H^n)$ for all $n\in\mathbb N$.
Is it true $\ind(\mathbb H)=2$? 
(Here $\mathbb H$~is the skew field of quaternions.) 
An analogous question about other Cayley--Dickson 
algebras.

\end{prbl}

\begin{prbl}\label{prb: ind polyring}
For which algebras (in particular, polyrings) 
the equality $\ind(K^n)=(\ind K)\cdot n$ does hold 
for all $n\in\mathbb N$? 
\end{prbl}

\begin{prbl}\label{prb: ind} 
What can be values of the function~$\ind$ for 
various algebras (skew fields, Abelian groups, groups, 
rings, polyrings, etc.)?
How does $\ind K^n$ depend on~$\ind K$?

\end{prbl}

\begin{prbl}\label{prb: dimension in hausdorff} 
Assume $K$~is a~topologizable polyring. Is then
$K$~topologizable by a~(Hausdorff) topology such 
that $K^n$~is closed nowhere dense subspace 
of~$K^{n+1}$ in the usual product topology
on~$K^{n+1}$?
\end{prbl}

\subsection{Other properties}

\begin{prbl}\label{prb: characterize properties}
Characterize groups (Abelian groups, rings, 
skew fields, etc.)~whose Zariski topology is:
\begin{itemize}
\setlength\itemsep{-0.39em}
\item[(i)] 
compact (locally compact, Lindel{\"o}f, 
paracompact, metrizable, etc.);
\item[(ii)] 
Hausdorff (Tychonoff, normal, etc.);
\item[(iii)]
zero-dimensional (totally disconnected, 
extremally disconnected);
\item[(iv)] 
a~base of a~filter (a~filter, an ultrafilter, 
a~filter of a~special kind, etc.) 
plus the empty set.
\end{itemize}
Which of these classes of algebras 
are first-order axiomatizable?
second-order axiomatizable?

\end{prbl}

Some similar questions were discussed in the 
literature. For any universal algebra~$A$ and 
$n\in\mathbb N$, the Zariski spaces of~$A^n$ 
are Noetherian iff each system of equations in 
$n$~variables has the same set of solutions 
that some its finite subsystem, 
see~\cite{Daniyarova et al 2008} 
(in fact, this reformulation is easy since 
Noetherian spaces can be charaterized as 
the spaces in which all subspaces are compact). 
$\Omega$-groups in which all Zariski closed sets are 
algebraic (thus any intersection of algebraic sets 
is algebraic) are characterized in~\cite{Lipyanski}.

\begin{prbl}\label{prb: preserving topol properties}
Establish properties of Zariski topologies that are 
preserved under: 
\begin{itemize}
\setlength\itemsep{-0.39em}
\item[(i)] 
homomorphic images (pre-images);
\item[(ii)] 
subalgebras (normal subgroups, ideals);
\item[(iii)]
direct products (reduced products, ultraproducts);
\item[(iv)] 
existential closure.
\end{itemize}
More generally, what is an interplay between 
the properties of a~given algebra and algebras 
obtained from it by these (or other) 
model-theoretic operations?

\end{prbl}

An interplay between some separability axioms 
and some model-theoretic operations on topological 
algebras is discussed in~\cite{Kearnes Sequeira}.

\begin{prbl}\label{prb: finite powers} 
Which properties of their Zariski topologies are 
shared by $A^n$ and~$B^n$ if $A$ and~$B$ are groups 
(divisible groups, rings, skew fields, etc.)~having 
the same equational theory? elementary equivalent?

\end{prbl}

\begin{prbl}\label{prb: zariski product} 
Characterize polyrings (and other algebras)~$K$ 
such that the Zariski space~$K^n$ coincides with the 
product space of $n$~copies of the Zariski space~$K$. 
\end{prbl}

%

\begin{prbl}\label{prb: finitely many roots}
Characterize skew fields (rings, Abelian groups, 
groups) in which any polynomial in one variable has 
only finitely many roots. (This is in fact 
Problem~\ref{prb: characterize properties}(iv) 
for the Fr{\'e}chet filter.)
Do such skew fields (algebraically closed skew 
fields) coincide with fields? 

\end{prbl}

\begin{prbl}\label{prb: metrizable polyring} 
Are all countable polyrings topologizable by 
a~metrizable topology? 
\end{prbl}

Recall that by Theorem~\ref{T3 polyring}, 
all countable polyrings are $T_3$-topologizable. 
Proposition~4.2 in~\cite{Dutka Ivanov} states that 
all countable first-order models in a~relational 
language are topologizable by a~metrizable 
zero-dimensional topology.

\subsection{Topological classification}

Let us say that a~topological space~$X$ is 
{\it locally homeomorphic\/} to a~space~$Y$ 
iff any open $U\subseteq X$ includes some 
open $V\subseteq U$ homeomorphic to~$Y$.
The following questions are formulated for 
fields as even in this simplest case answers 
may appear to be unclear. 


\begin{prbl}\label{prb: classify fields}

Let $K$~be an infinite field. Provide a~topological 
classification of connected definable subspaces of~$K^n$ 
up to: (i)~homeomorphism; (ii)~local homeomorphism.
(E.g., do $K^2$ and $K^2\setminus A$ homeomorphic
if $A$~is a~singleton? a~proper closed subset of~$K^2$?) 

The same task for skew fields (rings, Abelian groups, 
groups, etc.).

%
\end{prbl}

\begin{prbl}\label{prb: roots fields loc homeom}

Is it true that for any polynomial
$F\in K[x_1,\ldots,x_n]$ over a~field~$K$
the set of its roots (as a~subspace of~$K^n$) 
is locally homeomorphic to the space~$K^m$ for some~$m$? 
(This is clear for $n\le2$; for $n=3$ the simple instance
which seems unclear is a~$2$-dimensional sphere in~$K^3$.) 

If the answer is affirmative, what about skew fields
(rings, Abelian groups, groups)?

\end{prbl}

\begin{prbl}\label{prb: embedding field}
Let $K$~be a~field. Is any space locally homeomorphic
to the space~$K^n$ embeddable into the space~$K^m$ 
for some $m\in\mathbb N$ (e.g.~into~$K^{2n+1}$)? 

If yes, what about skew fields
(rings, Abelian groups, groups)?

\end{prbl}

\begin{prbl}\label{prb: homeom quaternion}
Is the skew field~$\mathbb H$ (with the Zariski topology
given by polynomials in one variable) homeomorphic to 
the complex plane~$\mathbb C^2$ (with the Zariski topology
given by polynomials in two variables)? 
An analogous question about~$\mathbb H^n$, 
other Cayley--Dickson algebras.

\end{prbl}

\subsection{Miscellaneous}

\begin{prbl}\label{prb: Zariski spaces of groups}
Characterize topological spaces~$X$ which are 
the Zariski spaces of groups. 

An analogous question about quasigroups, fields, 
rings, etc. If $X$~is the Zariski space of a~quasigroup,
is it the Zariski space of a~group?
\end{prbl}

\begin{prbl}\label{prb: spaces of groups}
Characterize Hausdorff topological spaces~$X$ which are 
spaces of groups, i.e.~such that there is a~group structure 
on~$X$ with continuous group operations in this topology. 

An analogous question about quasigroups, fields, 
rings, etc. If a~Hausdorff space~$X$ is a~space of 
a~quasigroup, is it a~space of a~group?
\end{prbl}

\begin{prbl}\label{prb: Zariski on ultrafilters}
For the ultrafilter extension of the additive semigroup 
$(\mathbb N,+)$, study the smallest topology on 
$\scc\mathbb N$ which includes its standard topology and: 
(i)~makes all left shifts continuous;
(ii)~includes the Zariski topology of $(\scc\mathbb N,+)$. 
(Notice that in case~(i) the addition of ultrafilters 
becomes separately topological, thus turning 
$(\scc\mathbb N,+)$ into a~quasi-topological semigroup, 
and that this topology is included into one defined 
in case~(ii).)

Analogous questions about $(\mathbb N,+,\,\cdot\,)$, 
about ultrafilter extensions of other groups, rings, etc.
\end{prbl}

\begin{prbl}\label{prb: reverse}
Determine the proof-theoretic strength of (arithmetic 
versions of) the Finite Products Theorem for 
$(\mathbb N,\,\cdot\,)$, the simultaneous Finite Sums 
and Products Theorem for $(\mathbb N,+,\,\cdot\,)$ 
(Theorem~\ref{addit-multipl} here), 
their multidimensional generalizations 
(variants of Theorem~\ref{hindman general}). 
\end{prbl}

\begin{prbl}\label{prb: restricted Zariski}
Study restricted Zariski topologies given by a~set 
of terms, e.g.~the set of terms of degrees~$\le n$. 
Characterize algebras having such topologies non-discrete
for a~given~$n$. 
\end{prbl} 

For the case of groups, such topologies were studied 
in~\cite{Banakh Protasov}.

\begin{prbl}\label{prb: Omega-groups}
Can some of results on the Zariski topology of polyrings 
in Section~2 be reproved for Abelian $\Omega$-groups? 
(Cf.~Remark~\ref{Abelian}.)
\end{prbl}

\begin{prbl}\label{prb: groupoid rings}
With regard to Lemma~\ref{lm: expansion}, 
what is an interplay between properties of: 
\begin{itemize}
\setlength\itemsep{-0.39em}
\item[(i)]
a~given Abelian group and the most expanded 
polyring over it?
\item[(ii)]
a~given universal algebra and the polyring 
of formal sums generated by it? 
\end{itemize}
\end{prbl}

\begin{prbl}\label{prb: polyrings}
Generalize theory of spectra of ring to polyrings, 
characterize arising spectral spaces.
\end{prbl}

\subsection*{Acknowledgement.} 
I~am grateful to 
Dikran Dikranjan, Neil Hindman, and Olga Sipacheva 
for their useful comments, 
both scientific and historical.



\bibliographystyle{plain}

\vskip+2em
\begin{footnotesize} 
\noindent
{\sc
Higher School of Modern Mathematics MIPT 
\/}
\\
{\it E-mail address:\/} 
d.i.saveliev@gmail.com 
\end{footnotesize} 

\hide{
\vskip+2em
\begin{footnotesize} 
\noindent
{\sc
The Russian Academy of Sciences, 
Steklov Mathematical Institute, 
Gubkina street~8, 
Moscow 119991 Russia
\/}
\\
{\it E-mail address:\/} 
d.i.saveliev@gmail.com 
\end{footnotesize} 
}

\end{document}